\documentclass{amsart}
\usepackage{amsmath}
\usepackage{amssymb}
\usepackage{amsthm}
\usepackage{graphicx}

\newtheorem{theorem}{Theorem}[section]

\usepackage[shortalphabetic,initials]{amsrefs}

\theoremstyle{definition}
\newtheorem{definition}[theorem]{Definition}
\newtheorem{remark}[theorem]{Remark}
\newtheorem{example}[theorem]{Example}
\newtheorem{construction}[theorem]{Construction}

\newtheorem{notation}[theorem]{Notation}
\newtheorem{question}[theorem]{Question}

\DeclareMathOperator{\lcm}{lcm}

\DeclareMathOperator{\gens}{gens}

\DeclareMathOperator{\mdeg}{mdeg}

\begin{document}
\title{Three simplicial resolutions}
\author{Jeffrey Mermin}
\address{Department of Mathematics, Oklahoma State University,
401 Mathematical Sciences, Stillwater, OK 74078}
\email{mermin@math.okstate.edu}
\urladdr{http://www.math.okstate.edu/$\sim$mermin/}
\date{}
\begin{abstract}We describe the Taylor and Lyubeznik resolutions as
  simplicial resolutions, and use them to show that the Scarf complex of
  a monomial ideal is the intersection of all its minimal free
  resolutions.
\end{abstract}
\maketitle

\section{Introduction}

Let $S=k[x_{1},\dots, x_{n}]$, and let $I\subset S$ be a monomial
ideal.  An important object in the study of $I$ is its \emph{minimal
  free resolution}, which encodes essentially all information about
$I$.  For example, the Betti numbers of $I$ can be read off as the
ranks of the modules in its minimal resolution.  

There are computationally intensive algorithms to compute the minimal
resolution of an arbitrary ideal (for example, in Macaulay 2
\cite{GS}, the command \texttt{res I} returns the minimal resolution of
$S/I$), 
but no general description is known, even for monomial ideals.  Thus,
it is an ongoing problem of considerable interest to find classes of
ideals whose minimal resolutions can be described easily.  A related
problem is to describe non-minimal resolutions which apply to large
classes of monomial ideals.  

The most general answer to the latter question is Taylor's resolution,
a (usually highly non-minimal) resolution which resolves an arbitrary
monomial ideal; it is discussed in Section \ref{s:Taylor}.

A very successful approach to both problems in the last decade has been to find
combinatorial or topological objects whose structures encode
resolutions in some way.  This approach began with \emph{simplicial
  resolutions} \cite{BPS}, and has expanded to involve polytopal
complexes \cites{NR, Si}, cellular complexes \cite{BS}, CW complexes
\cites{BW, Ve}, lattices \cites{Ma,Ph,PV}, posets \cite{Cl}, matroids
\cite{Tc}, and 
discrete Morse theory \cite{JW}.  

Resolutions associated to combinatorial objects have distinguished
bases, and relationships between the objects lead to relationships
between these bases in a very natural way.  It thus becomes possible
to compare and combine resolutions in all the ways that we can compare
or combine combinatorial structures.  For example, most of these
resolutions turn out to be subcomplexes of the Taylor resolution in a
very natural way.  The only new result in the paper is Theorem
\ref{thethm}, which describes the intersection of all the
simplicial resolutions of an ideal.

In section 2, we describe some background material and introduce
notation used throughout the paper.

Section 3 introduces the Taylor resolution in a way intended to
motivate simplicial resolutions, which are introduced in section 4.

Section 5 describes the Scarf complex, a simplicial complex which
often supports the minimal resolution of a monomial ideal, and
otherwise does not support any resolution.  

Section 6 defines the family of Lyubeznik resolutions.  This section
is essentially a special case of an 
excellent paper of Novik \cite{No}, which describes a more general
class of resolutions based on so-called ``rooting maps''.

Section 7 uses the Lyubeznik resolutions to prove Theorem
\ref{thethm}, that the Scarf complex of an ideal is equal to the
intersection of all its simplicial resolutions.

\textbf{Acknowledgements}  I thank Ananth Hariharan, Manoj Kummini,
Steve Sinnott, and the algebra groups at Cornell and Kansas for
inspiration and helpful discussions.

\section{Background and Notation}

Throughout the paper $S=k[x_{1},\dots, x_{n}]$ is a polynomial ring
over an arbitrary field $k$.  In the examples, we use the variables
$a,b,c,\dots$ instead of $x_{1},x_{2},x_{3},\dots$.

We depart from the standard notation in two ways, each designed to
privilege monomials.  First, we write the standard or ``fine''
multigrading multiplicatively, indexed by monomials, rather than
additively, indexed by $n$-tuples.  Second, we index our simplices by
monomials rather than natural numbers.  Details of both departures, as
well as some background on resolutions, are below.

\subsection{Algebra}

If $I\subset S$ is an ideal, then a \emph{free resolution} of $S/I$ is an
exact sequence
\[
\mathbb{F}:\dots \xrightarrow{\phi_{n}}F_{n}\xrightarrow{\phi_{n-1}}F_{n-1} \dots
\xrightarrow{\phi_{0}}F_{0}\to S/I\to 0
\]
where each of the $F_{i}$ is a free $S$-module.  

We say that $\mathbb{F}$ is \emph{minimal} if each of the modules
$F_{i}$ has minimum possible rank; in this case the ranks are the
\emph{Betti numbers} of $S/I$.  

It is not at all obvious \emph{a priori} that minimal resolutions
should exist.  For this reason, when $I$ is homogeneous, most standard
treatments take the 
following theorem as the definition instead:

\begin{theorem}\label{minimal}
Let $I$ be a homogeneous ideal, and let $\mathbb{F}$ be a resolution
of $S/I$.  Write $\mathbf{m}=(x_{1},\dots, x_{n})$.  Then $\mathbb{F}$
is minimal if and only if $\phi_{i}(F_{i+1})\subset \mathbf{m}F_{i}$
for all $i$.  
\end{theorem}

The proof of Theorem \ref{minimal} is technical; see, for example,
\cite{Pe}*{Section 9}.  

All the ideals we consider are homogeneous; in fact, they are
\emph{monomial} ideals, which is a considerably stronger property.

\begin{definition}  An ideal $I$ is a \emph{monomial ideal} if it has
  a generating set consisting of monomials.  There exists a unique minimal such
  generating set; we write $\gens(I)$ and call its elements the
  \emph{generators} of $I$.
\end{definition}

Monomial ideals respect a ``multigrading'' which refines the usual
grading.

\begin{notation}  We write the multigrading multiplicatively.  That
  is, for each monomial $m$ of $S$, set $S_{m}$ equal to the
  $k$-vector space spanned by $m$.  Then $S=\oplus S_{m}$, and
  $S_{m}\cdot S_{n}=S_{mn}$, so this decomposition is a grading.  We
  say that the monomial $m$ has \emph{multidegree} $m$.  We
  allow multidegrees to have negative exponents, so, for example, the
  twisted module $S(m^{-1})$ is a free module with generator in
  multidegree $m$, and $S(m^{-1})_{n}\cong S_{m^{-1}n}$ as a vector
  space; this is one-dimensional if no exponent of $m^{-1}n$ is negative,
  and trivial otherwise.  Note that $S=S(1)$.  
\end{notation}

If $N$ and $P$ are multigraded modules, we say that a map $\phi:N\to
P$ is \emph{homogeneous of degree} $m$ if 
$\phi(N_{n})\subset P_{mn}$ for all $n$, and that $\phi$ is simply
\emph{homogeneous} if it is homogeneous of degree $1$.  We say that a
resolution (or, more generally, an algebraic chain complex) is
\emph{homogeneous} if all its maps are homogeneous.

The minimal resolution of $S/I$ can be made homogeneous in a unique
way by assigning appropriate multidegrees to the generators of its
free modules; counting these generators by multidegree yields the
\emph{multigraded Betti numbers} of $S/I$.

\subsection{Combinatorics}

Let $M$ be a set of monomials (typically, $M$ will be the generators
of $I$).  The \emph{simplex on} $M$ is the set of all subsets of $M$;
we denote this by $\Delta_{M}$.  We will sometimes refer to the
elements of $M$ as \emph{vertices} of $\Delta_{M}$.

A \emph{simplicial complex} on $M$ is a subset of $\Delta_{M}$ which is
closed under the taking of subsets.  If $\Gamma$ is a
simplicial complex on $M$ and $F\in \Gamma$, we say that $F$ is a
\emph{face} of $\Gamma$.  Observe that if $F$ is a face of $\Gamma$
and $G\subset F$, then $G$ is also a face of $\Gamma$.  We require
that simplicial complexes be nonempty; that is, the empty set must
always be a face.  (In fact, for our purposes, we may as well assume
that every vertex must be a face.)  

If $F$ is a face of $\Gamma$, we assign $F$ the multidegree
$\lcm(m:m\in F)$.  Note that the vertex $m$ has multidegree $m$, and
that the empty set has multidegree $1$.  The \emph{order} of a face
$F$, written $|F|$, is the number of vertices in $F$; this is one
larger than its dimension.  If $G\subset F$ and $|G|=|F|-1$, we say
that $G$ is a \emph{facet} of $F$.

We adopt the convention that the unmodified word ``complex'' will
always mean an algebraic chain complex; simplicial complexes will be
referred to with the phrase ``simplicial complex''.  However, recall
that every simplicial complex is naturally associated to a chain
complex by the following standard construction from algebraic topology:

\begin{construction}\label{topochaincx}
Let $\Gamma$ be a simplicial complex on $M$, and impose an order on
the monomials of $M$ by writing $M=\{m_{1},\dots, m_{r}\}$.  Then we
associate to $\Gamma$ the chain complex $\mathbb{C}_{\Gamma}$ as
follows:  

For every face $F\in \Gamma$, we create a formal symbol $[F]$.
If we write $F=\{m_{i_{1}},\dots,
m_{i_{s}}\}$ with increasing indices $i_{j}$, then for each facet $G$
of $F$ we may write $G=F\smallsetminus\{m_{i_{j}}\}$ for some $j$; we
define an orientation by setting $\varepsilon_{G}^{F}$ equal to $1$ if
$j$ is odd and to $-1$ if $j$ is even.  For each $s$, let $C_{s}$ be
the $k$-vector space spanned by the symbols $[F]$ such that $|F|=s$,
and define the map 
\begin{align*}
\phi_{s-1}:C_{s}&\to C_{s-1}\\
[F]&\mapsto\sum_{G\text{ is a facet of }F}\varepsilon_{G}^{F}[G].
\end{align*}

Then we set $\mathbb{C}_{\Gamma}$ equal to the complex of vector spaces
\[
\mathbb{C}_{\Gamma}:0\to
C_{r}\xrightarrow{\phi_{r-1}}\dots\xrightarrow{\phi_{1}} C_{1} 
\xrightarrow{\phi_{0}} C_{0}\to 0.
\]
\end{construction}

The proof that $\mathbb{C}_{\Gamma}$ is a chain complex involves a
straightforward computation of $\phi^{2}([m_{i_{1}},\dots,
  m_{i_{s}}])$.  The (reduced) homology of $\Gamma$ is defined to be
the homology of this complex.

In section 4, we will replace this complex with a homogeneous complex
of free $S$-modules.

\section{The Taylor resolution}\label{s:Taylor}

Let $I=(m_{1},\dots,m_{s})$ be a monomial ideal.  The Taylor
resolution of $I$ is constructed as follows:

\begin{construction}\label{taylorconstruction}

For a subset $F$ of $\{m_{1},\dots, m_{r}\}$, set
$\lcm(F)=\lcm\{m_{i}:m_{i}\in F\}$.  For each such $F$, we define a
formal symbol $[F]$, called a \emph{Taylor symbol}, with multidegree
equal to $\lcm(F)$.  For each 
$i$, set $T_{i}$ equal to the free $S$-module with basis
$\{[F]:|F|=i\}$ given by the symbols corresponding to subsets of size
$i$.  Note that $T_{0}=S[\varnothing]$ is a free module of rank one,
and that all other $T_{i}$ are multigraded modules with generators in
multiple multidegrees depending on the symbols $[F]$.  

Define $\phi_{-1}:T_{0}\to S/I$ by $\phi_{-1}(f[\varnothing])=f$.
Otherwise, we construct $\phi_{i}:T_{i+1}\to T_{i}$ as follows.

Given $F=\{m_{j_{1}},\dots, m_{j_{i}}\}$, written with the indices in
increasing order, and $G=F\smallsetminus\{m_{j_{k}}\}$, we set the
\emph{sign} $\varepsilon_{G}^{F}$ equal to $1$ if $k$ is odd and to $-1$ if
$k$ is even.  Finally, we set 
\[
\phi_{F}=\sum_{G\text{ is a facet of }F}\varepsilon_{G}^{F}\frac{\lcm(F)}{\lcm(G)}[G],
\]

and define $\phi_{i}:T_{i+1}\to T_{i}$ by extending the various
$\phi_{F}$.  Observe that all of the $\phi_{i}$ are homogeneous with
multidegree $1$.  

The \emph{Taylor resolution} of $I$ is the complex
\[
\mathbb{T}_{I}:0\to T_{r}\xrightarrow{\phi_{r-1}}\dots\xrightarrow{\phi_{1}} T_{1} 
\xrightarrow{\phi_{0}} T_{0}\to S/I\to 0.
\]

It is straightforward to show that the Taylor resolution is a
homogeneous chain complex.
\end{construction}

The construction of the Taylor resolution is very similar to Construction
\ref{topochaincx}; in fact, if $\Gamma$ is the complete simplex, the
only difference is the presence of the lcms in the boundary maps.  We
will explore this connection in the next section.

\begin{example}
Let $I=(a,b^{2},c^{3})$.  Then the Taylor resolution of $I$ is 
\[
\mathbb{T}_{I}:0 \to S[a,b^{2},c^{3}]
\xrightarrow{\begin{pmatrix}a\\-b^{2}\\c^{3}\end{pmatrix}} 
\begin{array}{c}S[b^{2},c^{3}]\\ \oplus\\ S[a,c^{3}]\\ \oplus\\ 
S[a,b^{2}]\end{array} 
\xrightarrow{\begin{pmatrix}0 & -c^{3} & -b^{2}\\ -c^{3} & 0 & a\\ b^{2} & a
    & 0\end{pmatrix}} \begin{array}{c}S[a]\\ \oplus\\ S[b^{2}]\\ \oplus\\ 
S[c^{3}]\end{array}
\xrightarrow{\begin{pmatrix}a&b^{2}&c^{3}\end{pmatrix}} S[\varnothing]\to S/I \to 0. 
\]

Observe that $I$ is a complete intersection and $\mathbb{T}_{I}$ is
its Koszul complex.  In fact, these two complexes coincide for all
monomial complete intersections.  
\end{example}

\begin{example}\label{stdexample}
Let $I=(a^{2},ab,b^{3})$.  Then the Taylor resolution of $I$ is 
\[
\mathbb{T}_{I}:0 \to S[a^{2},ab,b^{3}]
\xrightarrow{\begin{pmatrix}a\\-1\\b^{2}\end{pmatrix}} 
\begin{array}{c}S[ab,b^{3}]\\ \oplus\\ S[a^{2},b^{3}]\\ \oplus\\ 
S[a^{2},ab]\end{array} 
\xrightarrow{\begin{pmatrix}0 & -b^{3} & -b\\ -b^{2} & 0 & a\\ a^{2} & a
    & 0\end{pmatrix}} \begin{array}{c}S[a^{2}]\\ \oplus\\ S[ab]\\ \oplus\\ 
S[b^{3}]\end{array}
\xrightarrow{\begin{pmatrix}a^{2}&ab&b^{3}\end{pmatrix}} S[\varnothing]\to S/I \to 0. 
\]
This is not a minimal resolution; the Taylor resolution is very rarely
minimal.   
\end{example}

\begin{theorem}\label{taylorthm}
The Taylor resolution of $I$ is a resolution of $I$.
\end{theorem}

It is not too difficult to show that $\phi^{2}=0$ in the Taylor
complex, but it is not at all clear from the construction that the
complex is exact.  This
seems to be most easily established indirectly by showing that the
Taylor resolution is a special case of some more general phenomenon.
We will prove Theorem \ref{taylorthm} in the next section, using the
language of simplicial resolutions.  Traditionally, one builds the
Taylor resolution as an iterated mapping cone; we sketch that argument
below.

\begin{proof}[Sketch of Theorem \ref{taylorthm}]
Write $I=(m_{1},\dots, m_{r})$,  and let $J=(m_{1},\dots, m_{r-1})$.
Consider the short exact sequence
\[
0\to \frac{S}{(J:m_{r})} \xrightarrow{m_{r}}
\frac{S}{J}\to \frac{S}{I}\to 0.
\]

If $(\mathbb{A},\alpha)$ and $(\mathbb{B},\beta)$ are free resolutions
of $S/(J:m_{r})$ and $S/J$, respectively, then multiplication by $m_{r}$
induces a map of complexes $(m_{r})_{*}:\mathbb{A}\to\mathbb{B}$.  The
\emph{mapping cone complex} $(\mathbb{T},\gamma)$ is defined by
setting $T_{i}=B_{i}\oplus A_{i-1}$ and $\gamma|_{\mathbb{B}}=\beta$,
$\gamma|_{\mathbb{A}}=(m_{r})_{*}-\alpha$; it is a free resolution of
  $S/I$ (see, for example, \cite{Pe}*{Section 27}).

Inducting on $r$, $S/J$ is resolved by the Taylor resolution on its
generators $\{m_{1},\dots, m_{r-1}\}$, and $S/(J:m_{r})$ is resolved
by the Taylor resolution on its (possibly redundant) generating set
$\{\frac{\lcm(m_{1},m_{r})}{m_{r}},\dots,
\frac{\lcm(m_{r-1},m_{r})}{m_{r}}\}$.  The resulting mapping cone is
the Taylor resolution of $I$.
\end{proof}

\section{Simplicial resolutions}

If $\Gamma=\Delta$, the construction of the Taylor resolution differs
from the classical topological construction of the chain complex
associated to $\Gamma$ only by the presence of the monomials
$\frac{\lcm(F)}{\lcm(G)}$ in its differential maps.  This observation
leads naturally to the question of what other simplicial complexes
give rise to resolutions in the same way.  The resulting resolutions
are called \emph{simplicial}.  Simplicial resolutions and, more
generally, resolutions arising from other topological structures (it
seems that the main results can be tweaked to work for anything
defined in terms of skeletons and boundaries) have proved to be an
instrumental tool in the understanding of monomial ideals.  We
describe only the foundations of the theory here; for a more detailed
treatment, the original paper of Bayer, Peeva, and Sturmfels
\cite{BPS} is a very readable introduction.

\begin{construction}\label{simplicialres}
Let $M$ be a set of monomials, and let $\Gamma$ be a simplicial
complex on $M$ (recall that this means that the vertices of $\Gamma$
are the monomials in $M$).  Fix an ordering on the elements of $M$;
this induces an orientation $\varepsilon$ on $\Gamma$.  Recall that
$\varepsilon_{G}^{F}$ is either $1$ or $-1$ if $G$ is a facet of $F$
(see Construction \ref{topochaincx} for the details); it is often
convenient to formally set $\varepsilon_{G}^{F}$ equal to zero when
$G$ is not a facet of $F$.  

We assign a multidegree to each face $F\in \Gamma$ by the rule
$\mdeg(F)=\lcm(m:m\in F)$ (recall that $F$ is a subset of $M$, so its
elements are monomials).  

Now for each face $F$ we create a formal symbol $[F]$ with multidegree
$\mdeg(F)$.  Let $H_{s}$ be the free module with basis $\{[F]:|F|=s\}$,
and define the differential
\begin{align*}
\phi_{s-1}:H_{s}&\to H_{s-1}\\
[F]&\mapsto\sum_{G\text{ is a facet of
  }F}\varepsilon_{G}^{F}\frac{\mdeg(F)}{\mdeg(G)}[G]. 
\end{align*}

The \emph{complex associated to} $\Gamma$ is then the algebraic chain
complex
\[
\mathbb{H}_{\Gamma}:0\to
H_{r}\xrightarrow{\phi_{r-1}}\dots\xrightarrow{\phi_{1}} H_{1} 
\xrightarrow{\phi_{0}} H_{0}\to S/I\to 0.
\]
\end{construction}

Construction \ref{simplicialres} differs from Construction
\ref{topochaincx} in that it is a complex of free $S$-modules rather
than vector spaces.  The boundary maps are identical except for the
monomial coefficients, which are necessary to make the complex
homogeneous.

\begin{example}
Let $I$ be generated by $M$, and let $\Delta$ be the simplex with
vertices $M$.  Then the Taylor resolution of $I$ is the complex
associated to $\Delta$.
\end{example}

\begin{example}\label{threeexamples}
Let $I$ be generated by $M=\{a^{2},ab,b^{3}\}$, and let $\Delta$ be
the full simplex on 
$M$, $\Gamma$ the simplicial complex
with facets $\{a^{2},ab\}$ and $\{ab,b^{3}\}$, and $\Theta$ the
zero-skeleton of $\Delta$.  These simplicial complexes, with their
faces labeled by multidegree, are pictured in figure \ref{thefigure}.

\begin{figure}[hbtf]
\includegraphics[width=5in]{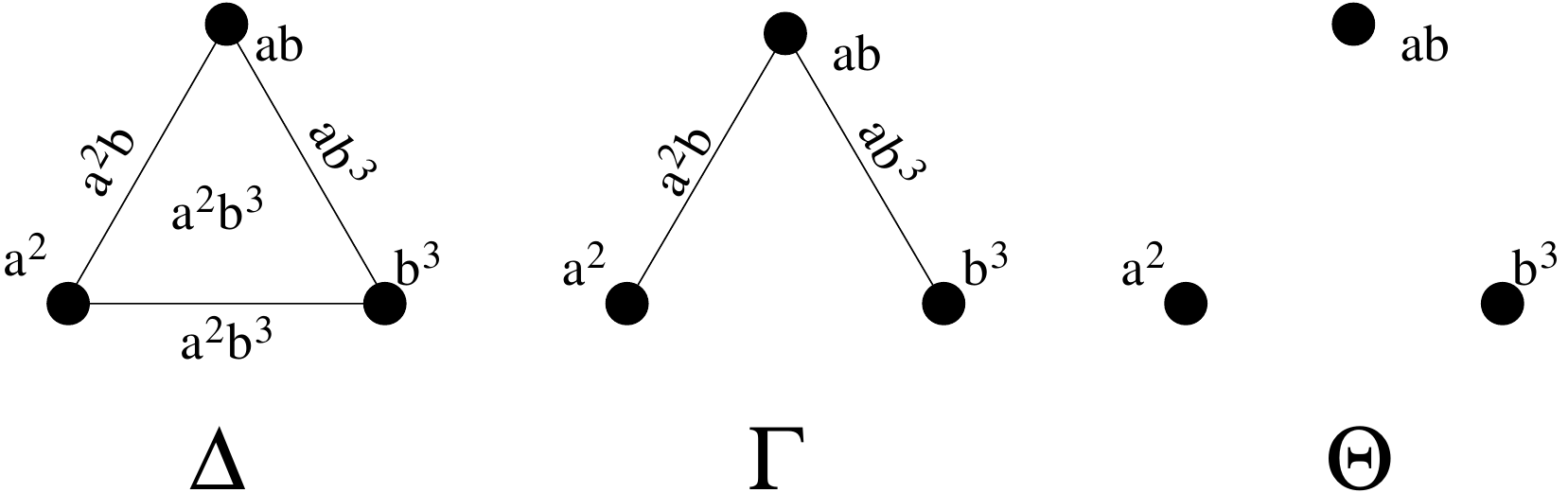}
\caption{The simplicial complexes $\Delta$, $\Gamma$, and $\Theta$ of
  example \ref{threeexamples}.}\label{thefigure}
\end{figure}

The algebraic complex associated to $\Delta$ is the Taylor resolution
of Example \ref{stdexample}.  The other two associated complexes are
\[
\mathbb{H}_{\Gamma}:0\to \begin{array}{c}S[a^{2},ab]\\ \oplus\\ 
  S[ab,b^{3}]\end{array}
  \xrightarrow{\begin{pmatrix}-b&0\\a&-b^{2}\\0&a\end{pmatrix}} 
  \begin{array}{c}S[a^{2}]\\ \oplus\\ S[ab]\\ \oplus\\
    S[b^{3}]\end{array}
  \xrightarrow{\begin{pmatrix}a^{2}&ab&b^{3}\end{pmatrix}}S[\varnothing]\to S/I \to 0
\]
and
\[
\mathbb{H}_{\Theta}:0\to \begin{array}{c}S[a^{2}]\\ \oplus\\ S[ab]\\ \oplus\\
    S[b^{3}]\end{array}
  \xrightarrow{\begin{pmatrix}a^{2}&ab&b^{3}\end{pmatrix}} S[\varnothing]\to S/I \to 0.
\]

$\mathbb{H}_{\Gamma}$ is a resolution (in fact, the minimal
resolution) of $S/I$, and $\mathbb{H}_{\Theta}$ is not a resolution of $I$.
\end{example}

The algebraic complex associated to $\Gamma$ is not always exact; that
is, it does not always give rise to a resolution of $I$.  When this
complex is exact, we call it a \emph{simplicial resolution}, or
\emph{the (simplicial) resolution supported on $\Gamma$}.  It turns
out that there is a topological condition describing whether
$\Gamma$ supports a resolution.

\begin{definition}  Let $\Gamma$ be a simplicial complex on $M$, and
  let $\mu$ be a multidegree.  We set $\Gamma_{\leq \mu}$ equal to the
  simplicial subcomplex of $\Gamma$ consisting of the faces with
  multidegree divisible by $\mu$,
  \[
  \Gamma_{\leq \mu}=\{F\in \Gamma:\deg(F) \text{ divides }\mu\}.
  \]
  Observe that $\Gamma_{\leq \mu}$ is precisely the faces of $\Gamma$
  whose vertices all divide $\mu$.  
\end{definition}

\begin{theorem}[Bayer-Peeva-Sturmfels]\label{criterion}  Let $\Gamma$
  be a simplicial 
  complex supported on $M$, and set $I=(M)$.  Then $\Gamma$ supports a
  resolution of $S/I$ if and only if, for all $\mu$, the simplicial complex
  $\Gamma_{\leq \mu}$ has no homology over $k$.
\end{theorem}

\begin{proof}
Since $\mathbb{H}_{\Gamma}$ is homogeneous, it is exact if and only if it is
exact (as a complex of vector spaces) in every multidegree.  Thus, it
suffices to examine the restriction of $\mathbb{H}_{\Gamma}$ to each
multidegree $\mu$.  

Observe that
$(S[F])_{\mu}\cong S(\frac{1}{\mdeg(F)})_{\mu}\cong S_{\frac{\mu}{\mdeg(F)}}$ is a
one-dimensional vector space with basis $\frac{\mu}{\mdeg(F)}$ if
$\mdeg(F)$ divides $\mu$, and is zero otherwise.  Furthermore, since
the differential maps $\phi$ are homogeneous, the monomials appearing
in their definition are precisely those which map these basis elements
to one another.  Thus $(\mathbb{H}_{\Gamma})_{\mu}$ is, with minor
abuse of notation, precisely the complex of vector spaces which arises
when computing (via Construction \ref{topochaincx}) the homology of
the simplicial complex $\{F\in \Gamma:\mdeg(F) \text{ divides
}\mu\}$, and this complex is $\Gamma_{\leq\mu}$.  

We conclude that $\Gamma$ supports a resolution of $I$ if and only if
$(\mathbb{H}_{\Gamma})_{\mu}$ is exact for every $\mu$, if and only if
$(\mathbb{H}_{\Gamma})_{\mu}$ has no homology for every $\mu$, if and
only if $\Gamma_{\leq\mu}$ has no homology for every $\mu$.
\end{proof}

\begin{example}
The simplicial complexes $\Gamma_{\leq \mu}$ depend on the
underlying monomials $M$, so that it is possible for a simplicial
complex to support a resolution of some monomial ideals but not
others.  For example, the simplicial complex $\Gamma$ in example
\ref{threeexamples}  supports a resolution of $I=(a^{2},ab,b^{3})$
because no monomial is divisible by $a^{2}$ and $b^{3}$ without also
being divisible by $ab$.  However, if we were to relabel the vertices
with the monomials $a$, $b$, and $c$, the resulting simplicial complex
$\Gamma'$ 
would not support a resolution of $(a,b,c)$ because $\Gamma'_{\leq
  ac}$ would consist of two points; this simplicial complex has
nontrivial zeroeth homology.
\end{example}

\begin{remark}
Note that the homology of a simplicial complex can depend on the
choice of field, so some simplicial complexes support resolutions 
over some fields but not others.  For example, if $\Gamma$ is a
triangulation of a torus, it may support a resolution if the field has
characteristic zero, but will not support a resolution in
characteristic two.  In particular, resolutions of monomial ideals can
be characteristic-dependent.
\end{remark}

Theorem \ref{criterion} allows us to give a short proof that the
Taylor resolution is in fact a resolution.

\begin{proof}[Proof of Theorem \ref{taylorthm}]  Let $\mu$ be given.
  Then $\Delta_{\leq\mu}$ is the simplex with vertices $\{m\in M:m
  \text{ divides }\mu\}$, which is either empty or 
  contractible.
\end{proof}

\section{The Scarf complex}

Unfortunately, the Taylor resolution is usually not minimal.  The
nonminimality is visible in the nonzero scalars in the differential
maps, which occur whenever there exist faces $F$ and $G$ with the same
multidegree such that $G$ is in the boundary of $F$.  It is tempting
to try to simply remove the nonminimality by removing all such faces;
the result is the Scarf complex.

\begin{construction}  Let $I$ be a monomial ideal with generating set
  $M$.  Let $\Delta_{I}$ be the full simplex on $M$, and let
  $\Sigma_{I}$ be the simplicial subcomplex of $\Delta_{I}$ consisting
  of the faces with unique multidegree, 
\[
\Sigma_{I}=\{F\in \Delta_{I}:\mdeg(G)=\mdeg(F)\implies G=F\}.
\]
We say that $\Sigma_{I}$ is the \emph{Scarf simplicial complex} of
$I$; the associated algebraic chain complex $\mathbb{S}_{I}$ is called
the \emph{Scarf complex} of $I$.  The multidegrees of the faces of
$\Sigma_{I}$ are called the \emph{Scarf multidegrees} of $I$.
\end{construction}

\begin{remark}  It is not obvious that $\Sigma_{I}$ is a simplicial
  complex.  Let $F\in \Sigma_{I}$; we will show that every subset of
  $F$ is also in $\Sigma_{I}$.  Suppose not; then there exists a
  minimal $G\subset F$ which shares a multidegree with some other
  $H\in \Delta_{I}$.  Let $E$ be the symmetric difference of $G$ and
  $H$.  Then the symmetric difference of $E$ and $F$ has the same
  multidegree as $F$.
\end{remark}

\begin{example}  Let $I=(a^{2},ab,b^{3})$.  Then the Scarf simplicial
  complex of $I$ is the complex $\Gamma$ in figure \ref{thefigure}.
  The Scarf complex of $I$ is the minimal resolution
\[
\mathbb{S}_{I}:0\to \begin{array}{c}S[a^{2},ab]\\ \oplus\\ 
  S[ab,b^{3}]\end{array}
  \xrightarrow{\begin{pmatrix}-b&0\\a&-b^{2}\\0&a\end{pmatrix}} 
  \begin{array}{c}S[a^{2}]\\ \oplus\\ S[ab]\\ \oplus\\
    S[b^{3}]\end{array}
  \xrightarrow{\begin{pmatrix}a^{2}&ab&b^{3}\end{pmatrix}}S[\varnothing]\to S/I \to 0.
\]
\end{example}

\begin{example}\label{notscarf}
Let $I=(ab,ac,bc)$.  The Scarf simplicial complex of $I$ consists of
three disjoint vertices.  The Scarf complex of $I$ is the complex
\[
\mathbb{S}_{I}:0\to \begin{array}{c}S[ab]\\ \oplus\\ S[ac]\\ \oplus\\
    S[bc]\end{array}
  \xrightarrow{\begin{pmatrix}ab&ac&bc\end{pmatrix}} S[\varnothing]\to S/I \to 0.
\]
It is not a resolution.
\end{example}

Example \ref{notscarf} shows that not every monomial ideal is resolved by its
Scarf complex.  We say that a monomial ideal is \emph{Scarf} if its
Scarf complex is a resolution.  

\begin{theorem} If the Scarf complex of $I$ is a resolution, then it
  is minimal.
\end{theorem}
\begin{proof}
By construction, no nonzero scalars can occur in the differential
matrices.
\end{proof}

Bayer, Peeva and Sturmfels \cite{BPS} call an ideal \emph{generic} if
no variable 
appears with the same nonzero exponent in more than one generator.
They show that these ``generic'' ideals are Scarf.

Unfortunately, most interesting monomial ideals are not Scarf.
However, Scarf complexes have proved an important tool in constructing
ideals whose resolutions misbehave in various ways \cite{Ve}.

\begin{theorem}\label{scarfinside} Let $\mathbb{F}$ be a minimal
  resolution of $I$.  Then 
  the Scarf complex of $I$ is a subcomplex of $\mathbb{F}$.
\end{theorem}
\begin{proof}  This is \cite{Pe}*{Proposition 59.2}.  The proof requires
  a couple of standard facts about resolutions, but is otherwise
  sufficiently reliant on the underlying 
  simplicial complexes that we reproduce it anyway.

  We know (see, for example, \cite{Pe}*{Section 9}) that
  there is a homogeneous inclusion of complexes from $\mathbb{F}$ to the
  Taylor complex $\mathbb{T}$.  We also know that the multigraded
  Betti numbers of $I$, which count the generators of $\mathbb{F}$,
  can be computed from the homology of the simplicial complexes
  $\Delta_{\lneqq m}$ (\cite{Pe}*{Section 57}).  If $m=\mdeg(G)$ is a
  Scarf multidegree, then $b_{|G|,m}(S/I)=1$ and $b_{i,m}(S/I)=0$ for
    all other $I$.  If $m$ divides a Scarf multidegree but is not
    itself a Scarf multidegree, then $b_{i,m}(S/I)=0$ for all $i$.  In
    particular, when $m$ is a Scarf multidegree, the Betti numbers of
    multidegree $m$ also count the number of faces of multidegree $M$
    in both $\Delta_{I}$ and $\Sigma_{I}$; these numbers are never
    greater than one.
    
    By induction on multidegrees, each generator of
    $\mathbb{F}$ with a Scarf multidegree must (up to a scalar) be
    mapped under the inclusion to the unique generator of the Taylor
    resolution with the same multidegree.  However, these are exactly the
    generators of the Scarf complex.  Thus, the inclusion from
    $\mathbb{F}$ to $\mathbb{T}$ induces an inclusion from
    $\mathbb{S}$ to $\mathbb{F}$.
\end{proof}

\section{The Lyubeznik resolutions}

If the Taylor resolution is too large, and the Scarf complex is too
small, we might still hope to construct simplicial resolutions
somewhere in between.  Velasco \cite{Ve} shows that it is impossible to get the
minimal resolution of every ideal in this way, even if we replace
simplicial complexes with much more general topological objects.
However, there are still classes of simplicial resolutions which are
in general much smaller than the Taylor resolution, yet still manage
to always be resolutions.  One such class is the class of Lyubeznik
resolutions, introduced below.

Our construction follows the treatment in an excellent paper of Novik
\cite{No}, which presents the Lyubeznik resolutions as special cases
of resolutions arising from ``rooting maps''.  The only difference
between the following construction and Novik's paper is that the extra
generality has been removed, and the notation is correspondingly
simplified.

\begin{construction}
Let $I$ be a monomial ideal with generating set $M$, and fix an
ordering $\prec$ on the monomials appearing in $M$.  (We do not require
that $\prec$ have any special property, such as a term order; any total
ordering will do.)  Write $M=\{m_{1},\dots, m_{s}\}$ with $m_{i}\prec
m_{j}$ whenever $i<j$.  

Let $\Delta_{I}$ be the full simplex on $M$; for a monomial $\mu\in
I$, set $\min(\mu)=\min_{\prec}\{m_{i}:m_{i}\text{ divides }\mu\}$.  For
a face $F\in
\Delta_{I}$, set $\min(F)=\min(\mdeg(F))$.  Thus $\min(F)$ is a
monomial.  We expect that in fact 
$\min(F)$ is a vertex of $F$, but this need not be the case:  for
example, if $F=\{a^{2},b^{2}\}$, we could have $\min(F)=ab$.  

We say that a face $F$ is \emph{rooted} if every nonempty subface
$G\subset F$ satisfies $\min(G)\in G$.  (Note that in particular
$\min(F)\in F$.)  By construction, the set $\Lambda_{I,\prec}=\{F\in
\Delta_{I}: F \text{ is rooted}\}$ is a simplicial complex; we call it
the \emph{Lyubeznik simplicial complex} associated to $I$ and
$\prec$.  The associated algebraic chain complex $\mathbb{L}_{I,\prec}$
is called a \emph{Lyubeznik resolution} of $I$.
\end{construction}

\begin{example}  Let $I=(ab,ac,bc)$.  Then there are three distinct
  Lyubeznik resolutions of $I$, corresponding to the simplicial
  complexes pictured in Figure \ref{abacbc}:  
\begin{figure}[hbtf]
\includegraphics[width=5in]{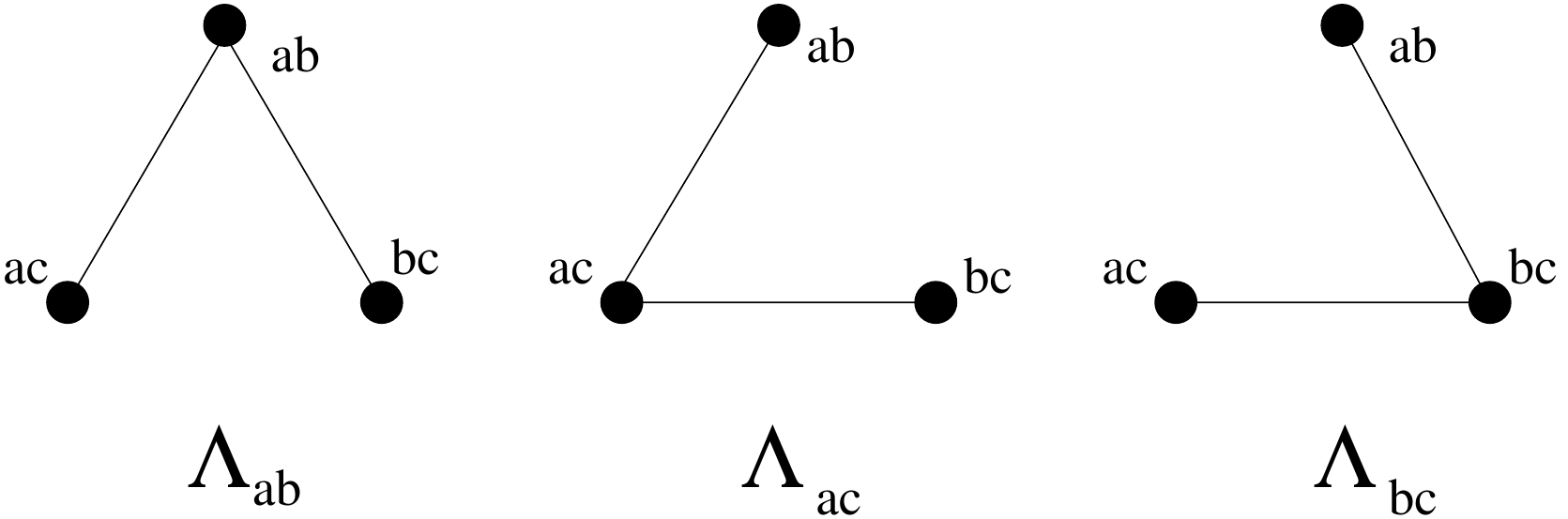}
\caption{The Lyubeznik resolutions of $I=(ab,ac,bc)$.}\label{abacbc}
\end{figure}
  $\Lambda_{ab}$ arises
  from the orders $ab\prec ac\prec bc$ and $ab\prec bc\prec ac$,
  $\Lambda_{ac}$ arises from the orders with $ac$ first, and
  $\Lambda_{bc}$ arises from the orders with $bc$ first.  Each of
  these resolutions is minimal.
\end{example}

\begin{example}\label{lyubex2}
Let $I=(a^{2},ab,b^{3})$.  There are two Lyubeznik resolutions of $I$:
the Scarf complex, arising from the two orders with $ab$ first, and
the Taylor resolution, arising from the other four orders.  The
corresponding simplicial complexes are pictured in figure
\ref{lyubstd}.
\begin{figure}[hbtf]
\includegraphics[width=3in]{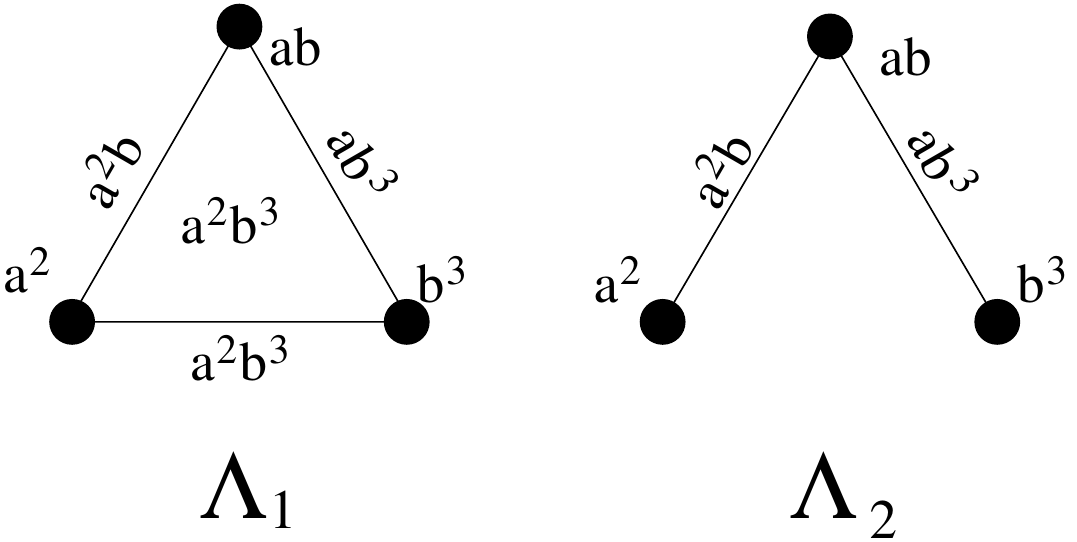}
\caption{The Lyubeznik resolutions of $I=(a^{2},ab,b^{3})$.}\label{lyubstd}
\end{figure}
\end{example}

\begin{remark}
It is unclear how to choose a total ordering on the generators of $I$
which produces a smaller Lyubeznik resolution.  Example \ref{lyubex2}
suggests that the obvious choice of a term order is a bad one:
the lex and graded (reverse) lex orderings all yield the Taylor
resolution, while the minimal resolution arises from orderings which
cannot be term orders.
\end{remark}

We still need to show that, unlike the Scarf complex, the Lyubeznik
resolution is actually a resolution.

\begin{theorem} The Lyubeznik resolutions of $I$ are resolutions.
\end{theorem}
\begin{proof} Let $M=(m_{1},\dots, m_{s})$ be the generators of $I$
  and fix an order 
  $\prec$ on $M$.  For each multidegree $\mu$, we need to show that
  the simplicial subcomplex $(\Lambda_{I,\prec})_{\leq \mu}$, consisting
  of the rooted faces with multidegree dividing $\mu$, has no
  homology.

  If $\mu\not\in I$, this is the empty complex.  If $\mu\in I$, we
  claim that $(\Lambda_{I,\prec})_{\leq \mu}$ is a cone.

  Suppose without
  loss of generality that $m_{1}=\min(\mu)$.  We claim that, if $F$ is a
  face of $(\Lambda_{I,\prec})_{\leq \mu}$, then $F\cup \{x_{1}\}$ is a
  face as well.  First, note that $\mdeg(F\cup \{x_{1}\})$ divides
  $\mu$ because both $x_{1}$ and $\mdeg(F)$ do.  Thus it suffices
  to show that $F\cup \{x_{1}\}$ is rooted.
  Observe that $\min(F\cup \{x_{1}\})=x_{1}$
  because $\mdeg(F\cup \{x_{1}\})$ divides $\mu$ and $x_{1}$ divides
  $\mdeg(F\cup \{x_{1}\})$.  If $G\subset F$, then $\min(G)\in G$
  because $F$ is rooted, and $\min(G\cup \{x_{1}\})=x_{1}$.  Thus
  $F\cup \{x_{1}\}$ is rooted.

  Hence $(\Lambda_{I,\prec})_{\leq \mu}$ is a simplicial cone on $x_{1}$
  and is contractible.
\end{proof}

\section{Intersections}

The only new result of this paper is that the Scarf complex of an
ideal $I$ is the intersection of all its minimal resolutions.  To make
this statement precise, we need to refer to some ambient space that
contains all the minimal resolutions; the natural choice is the Taylor
resolution.

\begin{theorem}\label{thethm}
Let $I$ be a monomial ideal.  Let $\mathbb{D}_{I}$ be the intersection
of all isomorphic 
embeddings of the minimal resolution of $I$ in its Taylor resolution.
Then $\mathbb{D}_{I}=\mathbb{S}_{I}$
is the Scarf complex of $I$.
\end{theorem}
\begin{proof}
We showed in Theorem \ref{scarfinside} that the Scarf complex is
contained in this intersection.  It suffices to show that the
intersection of all minimal resolutions lies inside the Scarf
complex.  We will show that in fact the intersection of all the
Lyubeznik resolutions is the Scarf complex.

Suppose that $F$ is a face of every Lyubeznik simplicial complex.
This means that, regardless of the ordering of the monomial generators
of $I$, the first generator dividing $\mdeg(F)$ appears in $F$.
Equivalently, every generator which divides $\mdeg(F)$ appears as a
vertex of
$F$. Thus, $F$ is the complete simplex on the vertices with
multidegree dividing $\mdeg(F)$.

Now suppose that there exists
another face $G$ 
with the same multidegree as $F$.  Every vertex of $G$ divides
$\mdeg(G)=\mdeg(F)$, so in particular $G\subset F$.  But this means
that $G$ is also a face of every Lyubeznik simplicial complex, so
every generator dividing $\mdeg(G)$ is a vertex of $G$ by the above
argument.  In particular, $F=G$.  This proves that $F$ is the unique
face with multidegree $\mdeg(F)$, i.e., $F$ is in the Scarf complex.
\end{proof}

\section{Questions}

The viewpoint that allows us to consider the statement of Theorem
\ref{thethm} requires that we consider a resolution together with its
basis, so resolutions which are isomorphic as algebraic chain
complexes can still be viewed as different objects.  The common use of
the phrase ``\emph{the} minimal resolution'' (instead of ``\emph{a}
minimal resolution'') suggests that this this point of view is
relatively new, or at any rate has not been deemed significant.  In
any event, there are some natural questions which would not make sense
from a more traditional point of view.

\begin{question}
Let $I$ be a monomial ideal.  Are there (interesting) resolutions of
$I$ which are not subcomplexes of the Taylor resolution?
\end{question}

All the interesting resolutions I understand are subcomplexes of the
Taylor complex in a very natural way:  their basis elements can be expressed
with relative ease as linear combinations of Taylor symbols.  If a
resolution is a subcomplex of the Taylor resolution, then it is
simplicial if and only if all its basis elements are Taylor symbols.
For a simplicial resolution to fail to be a subcomplex of the Taylor
complex, the set of vertices of its underlying simplicial complex must
not be a subset of the generators - in other words, the underlying
presentation must not be minimal.  

\begin{question}
Let $I$ be a monomial ideal.  Are there (interesting) resolutions of
$I$ with non-minimal first syzygies?
\end{question}

My suspicion is that such resolutions may exist, at least for special
classes of ideals, and may be useful in the 
study of homological invariants such as regularity which are
interested in the degree, rather than the number, of generators.

\end{document}